\documentclass[11pt]{amsart}

\usepackage{amsmath, amsfonts, amssymb, amsthm, hyperref, mathrsfs, xcolor}
\usepackage{upgreek}
\usepackage{wasysym}
\usepackage{mathrsfs}
\usepackage{ctable, tabularx}
\usepackage{fullpage}

\newtheorem{thm}{Theorem}[section]
\newtheorem*{thm*}{Theorem}

\newtheorem{lem}[thm]{Lemma}
\newtheorem{pro}[thm]{Proposition}

\newtheorem*{que}{Question}
\theoremstyle{definition}

\theoremstyle{remark}
\newtheorem{rmk}[thm]{Remark}

\numberwithin{equation}{section}

\newcommand{\mt}{J}

\def\a{\alpha}

\newcommand{\G}{\Gamma}

\newcommand{\comment}[1]{}

\newcommand{\gt}[1]{\mathfrak{#1}}

\newcommand{\SL}{\operatorname{\textsl{SL}}} 
\newcommand{\Aut}{\operatorname{Aut}}
\newcommand{\chr}{\operatorname{char}}
\newcommand{\genus}{\operatorname{genus}}
\newcommand{\val}{v}
\newcommand{\xmod}{{\rm \;mod\;}}
\newcommand{\moa}{{m}}
\newcommand{\cns}[1]{c_{#1}^*}
\newcommand{\jj}{{t}}
\newcommand{\Xp}{X_p}
\newcommand{\ex}{{e}}

\newcommand{\MM}{\mathbb{M}}	%monster group
\newcommand{\RR}{{\mathbb R}}	%Reals
\newcommand{\ZZ}{{\mathbb Z}}	%Integers
\newcommand{\QQ}{{\mathbb Q}}	%Rationals
\newcommand{\FF}{{\mathbb F}}	%finite field
\newcommand{\HH}{{\mathbb H}}	%half plane

\begin{document}

%--------------------------------------------------------------------------------------------------------------------------------%
\title{Modular Functions and the Monstrous Exponents}
%--------------------------------------------------------------------------------------------------------------------------------%

\renewcommand{\thefootnote}{\fnsymbol{footnote}} 

\author{John F.\ R.\ Duncan}
\address{Institute of Mathematics, Academia Sinica, Taipei, Taiwan}
\email{jduncan@as.edu.tw}

\author{Holly Swisher}
\address{Department of Mathematics, Oregon State University, Corvallis, Oregon, U.S.A.}
\email{swisherh@oregonstate.edu}

\begin{abstract}
We present a modular function-based approach to explaining, for primes larger than $3$, the exponents that appear in the prime decomposition of the order of the monster finite simple group.
\end{abstract}

\maketitle

%--------------------------------------------------------------------------------------------------------------------------------%
\section{Introduction}\label{sec:int}
%--------------------------------------------------------------------------------------------------------------------------------%

Compelling evidence for a new finite simple group, a group-theoretic ``friendly giant'' \cite{MR671653}, with order 
\begin{gather}\label{eqn:ord}
\#\MM = 2^{46} \cdot 3^{20}\cdot 5^9 \cdot 7^6\cdot 11^2\cdot 13^3 \cdot 17 \cdot 19 \cdot 23 \cdot 29 \cdot 31\cdot 41\cdot 47\cdot 59\cdot 71,
\end{gather}
was uncovered by Fischer and Griess, independently, in November of 1973 (see \cite[\S~15]{MR671653}).
It was called the {\em monster} by Conway, and it was conjectured (see \cite[\S~1]{MR554399} and \cite{MR0399248}) that it should have an irreducible representation of dimension $196883 = 47\cdot 59\cdot 71$.
On the basis of this the entire $194\times 194$ character table of the monster was computed (see \cite{MR827219} and also \cite[\S~15]{MR671653}) by Fischer, Livingstone and Thorne.

Then, in November of 1978 (see \cite[Prologue]{MR2215662}), McKay noticed that $196884=1+196883$ is the coefficient of $q$ in the $q$-expansion of the $j$-function, and suggested to Thompson that this might be significant. 
After further similar coincidences \cite{MR554402} were revealed\footnote{McKay mentions in his CRM-Fields Prize Lecture (see \cite{McKay_CRMFieldsPrizeLecture_Audio}) that it was Bernd Fischer's daughter (name unknown to the authors) that first looked at relating coefficients of $j$ to the monster, beyond the coefficient of $q$.},
Thompson proposed (see \cite{MR554401} and \cite[\S~1]{MR554399}) that there should be an infinite-dimensional graded $\MM$-module $V = \bigoplus_{n\geq -1} V_n$ that explains them.
The {\em McKay--Thompson series}
\begin{gather}
\label{eqn:sumtrgVnqn}
T_g(\tau):=\sum_{n\geq -1} {\rm tr}(g| V_n) q^n
\end{gather}
(here $q=e^{2\pi i \tau}$ for $\tau \in \HH$) of the conjectural monster module $V$ were completely prescribed by Conway and Norton's {\em monstrous moonshine} conjectures \cite{MR554399}.
Most strikingly, they predicted that each $T_g$ should be the normalised Hauptmodul for a specific group $\G_g \subseteq \text{SL}_2(\RR)$ of genus zero. 

It was January of 1980 (see \cite{MR605419} and \cite[\S~1]{MR671653}) when Griess established the existence of $\MM$ by constructing it as automorphisms of a $196884$-dimensional (commutative but non-associative) algebra. 
It was November of 1983 when Frenkel, Lepowsky, and Meurman announced \cite{FLMBerk} (see also \cite{FLMPNAS,FLM}) the construction of a graded $\MM$-module $V^\natural= \bigoplus_{n\geq -1} V^\natural_n$ such that 
\begin{gather}\label{eqn:sumtrgVnaturalnqn}
T_g^\natural(\tau) := \sum_{n\geq -1} {\rm tr}(g| V^\natural_n)q^n
\end{gather}
coincides with $T_g(\tau)$ in \eqref{eqn:sumtrgVnqn} for various $g$, and in particular for $g=e$ the identity. 
Then, in 1992, Borcherds \cite{MR1172696} resolved the monstrous moonshine conjectures by ingeniously showing that $T^\natural_g=T_g$ for all $g\in \MM$.

In recent years the story of moonshine has expanded to include other finite simple groups and other types of modular objects, including Jacobi forms \cite{MR3271175,MR2802725}, skew-holomorphic Jacobi forms \cite{pmo,MR3521908}, Siegel modular forms \cite{MR2985326,MR4865807}, and mock modular forms \cite{MR3449012,MR3433373,MR3539377}. 
Moreover, relationships between graded representations of finite groups and invariants of arithmetic geometry have appeared \cite{MR4029712,MR4611829,MR4906047,2017NatCo...8..670D,MR4291251,MR4230542}.

However, we are interested in revisiting a geometric perspective that originated earlier, before any connection to coefficients of $j$ was proposed:
In January of 1975, Ogg observed \cite{MR417184} that the primes $p$ that divide $\#\MM$ are the same as those for which the characteristic-$p$ supersingular $j$-invariants are all defined over $\FF_p$, rather than only over $\FF_{p^2}$.  
This provides a geometric characterisation of the primes that appear in (\ref{eqn:ord}).
Motivated by the possibility that there might be more to learn from this point of view, we consider the following question in this work.
\begin{que}
{\em Can we characterise not only which primes divide $\#\MM$, but also their multiplicities?}
\end{que}

We offer two answers to this question, for the primes greater than $3$.
Both take inspiration from Deligne's theorem on the $p$-adic rigidity of the $j$-function (see \cite{MR0294346,MR0437461}). 
One is formulated in terms of modular functions of prime-power level, and the other is formulated purely in terms of supersingular elliptic curves (or purely in terms of $j$, see Remark \ref{rmk:alt}). 
Our methods don't quite work for the cases $p=2,3$. 
We leave the treatment of these exceptional primes to future work.

To state our results define $J_N$ to be the normalised Hauptmodul for $\Gamma_0(N)$ if $\Gamma_0(N)$ has genus zero, and set $J_N:=0$ otherwise.
Define $J_{N+}$ analogously for $\Gamma_0(N)^+$.
Further, define $\gt{S}_p$ to be the set of $j$-invariants of supersingular elliptic curves in characteristic $p$, and write 
\begin{gather}\label{eqn:gtSp1gtSp2}
\gt{S}_p = \gt{S}_p^1 \sqcup \gt{S}_p^2
\end{gather} 
for the decomposition of $\gt{S}_p$ into those $j$-invariants ($\gt{S}_p^1$) that lie in $\FF_p$, and those ($\gt{S}_p^2$) that lie in $\FF_{p^2}$ but not in $\FF_p$.
Also define $\moa_p$ to be the minimum order of an automorphism group of a supersingular elliptic curve in characteristic $p$,
\begin{gather}\label{eqn:int:moap}
	\moa_p :=\min\left\{ \#\Aut(E) \mid E/K \text{ supersingular, } \chr(K) = p \right\}.
\end{gather}
See (\ref{eqn:moap}) for the value of $\moa_p$ in case $\#\gt{S}_p^2=0$.
We have $\moa_p=2$ in case $\gt{S}_p^2$ is not empty.

Write $\val_p(x)$ for the $p$-adic valuation of $x$.  
For $f$ a $q$-series with integer coefficients, $\val_p(f)$ denotes the minimum $p$-adic valuation of the coefficients of $f$.  
We prove the following two characterisations of $\val_p(\#\MM)$, for $p>3$.
\begin{thm}\label{thm:main}
Let $p$ be a prime greater than $3$.
Then we have
\begin{gather}\label{eqn:main}
\val_p(\#\MM) = \val_p(J_1-J_{p+}) + \val_p(J_1 - J_p) + \val_p(J_1 - J_{p^2}). 
\end{gather}
\end{thm}

\begin{thm}\label{thm:main-3d}
Let $p$ be a prime greater than $3$.
Then we have
\begin{gather}\label{eqn:main-3d}
\val_p(\#\MM) = 
\begin{cases}
	\frac32 \moa_p&\text{ if $\#\gt{S}_p^1=1$ and $\#\gt{S}_p^2=0$,}\\
	\frac12 \moa_p&\text{ if $\#\gt{S}_p^1>1$ and $\#\gt{S}_p^2=0$,}\\
	0& \text{ if $\#\gt{S}_p^2>0$.}
\end{cases}
\end{gather}
\end{thm}

\begin{rmk}\label{rmk:dsc}
We compute the right-hand side of (\ref{eqn:main}) for all $p$ in \S~\ref{sec:proof}, so we can detail the discrepancy in (\ref{eqn:main}-\ref{eqn:main-3d}) for $p=2,3$:
For $p=2$ the right-hand sides of (\ref{eqn:main}-\ref{eqn:main-3d}) both evaluate to $36$, whereas $\val_2(\#\MM)=46$, while for $p=3$ the right-hand sides of (\ref{eqn:main}-\ref{eqn:main-3d}) both evaluate to $18$, whereas $\val_3(\#\MM)=20$.
In particular, the right-hand sides of (\ref{eqn:main}-\ref{eqn:main-3d}) agree for all $p$.
\end{rmk}

\begin{rmk}\label{rmk:alt}
We may alternatively formulate Theorem \ref{thm:main-3d} purely in terms of the $j$-function, because it turns out that $\moa_p = \val_p (j|V_p - \Phi_p(j))$ for all $p$ prime, for $\Phi_p$ the $p$-th Faber polynomial. 
\end{rmk}

\begin{rmk}\label{rmk:frc}
Ogg \cite{MR417184} showed that $\Gamma_0(p)^+$ has genus zero exactly when $\gt{S}_p^2$ is empty.  
It develops (see Proposition \ref{pro:J1Jp} and the proof of Theorem \ref{thm:main-3d}) that for such $p$ we have
\begin{gather}\label{eqn:wppgenus0}
\val_p(J_1-J_{p+}) = \frac12\moa_p=\left\lceil\frac{12}{p-1}\right\rceil.
\end{gather}
On a similar note, $\Gamma_0(p)$ has genus zero precisely when $\gt{S}_p$ is a singleton, which is precisely when $p-1$ divides $12$.
For such $p$ we show (see Proposition \ref{pro:J1Jp}) that 
\begin{gather}\label{eqn:wpgenus0}
\val_p(J_1-J_{p}) = \frac{12}{p-1} + \left\lceil \frac{12}{p+1}\right\rceil.
\end{gather}
We can similarly treat the last term in (\ref{eqn:main}): The group $\Gamma_0(p^2)$ has genus zero just for the $p$ such that $p^2-1$ divides $24$, and for these primes (see Proposition \ref{pro:J1Jp2}) we have
\begin{gather}\label{eqn:wp2genus0}
\val_p(J_1-J_{p^2}) = \frac{24}{p^2-1}.
\end{gather}
\end{rmk}

The remainder is organised as follows. 
In Section \ref{sec:pre} we prove lemmas about modular functions and 
Hauptmoduls which enable our later analysis.  In Section \ref{sec:Deligne} we use a result of Deligne to determine the $p$-divisibility of $J_1|U_p$.  In Section \ref{sec:proof} we prove Theorems \ref{thm:main} and \ref{thm:main-3d}.
We acknowledge here that many of the arguments we present may be replaced by brute-force computation. 
We prefer to present a more organic and self-contained approach, in part because this relieves the reader from a computational burden, but also because we are motivated by the grand challenge of explaining how the monster arises from more natural mathematical objects.
We may not know yet what arguments will achieve this goal, but we can probably expect the more general ones to move us closer.

%--------------------------------------------------------------------------------------------------------------------------------%
\section*{Acknowledgement}
%--------------------------------------------------------------------------------------------------------------------------------%

J.D.\ gratefully acknowledges support from Academia Sinica (AS-IA-113-M03) and the National Science and Technology Council of Taiwan (112-2115-M-001-006-MY3).
H.S.\ was supported by the U.S. National Science Foundation grant DMS-2101906.

%--------------------------------------------------------------------------------------------------------------------------------%
\section{Preparations}\label{sec:pre}
%--------------------------------------------------------------------------------------------------------------------------------%

\subsection{Groups and functions} \label{sec:modfns}
We first recall that the {\em modular group} $\G_0(1):=\SL_2(\ZZ)$ is generated by $W_1:=\left( \begin{smallmatrix} 0 & -1 \\ 1 & 0\end{smallmatrix} \right)$ and $T:=\left( \begin{smallmatrix} 1 & 1 \\ 0 & 1\end{smallmatrix} \right)$.  
Next, for $N$ a positive integer, we set $V_N := \frac{1}{\sqrt{N}}\left( \begin{smallmatrix} N & 0 \\ 0 & 1\end{smallmatrix} \right)$, and define the {\em Fricke involution} 
\begin{gather}\label{eqn:WN}
W_N:=W_1V_{N} = V_N^{-1}W_1 = \frac{1}{\sqrt{N}}
\left(
\begin{smallmatrix}
	0 & -1 \\
	N & 0 
\end{smallmatrix} 
\right).
\end{gather}  
The $N$-th {\em Hecke congruence subgroup} is $\G_0(N) :=\G_0(1)\cap V_N^{-1}\G_0(1)V_{N^{}}$.
We also use $\G_0(N)^+ :=\langle \G_0(N),W_N\rangle = \G_0(N)\cup \G_0(N)W_N$.

For $\G=\G_0(N)$ or $\G=\G_0(N)^+$ the {\em cusps} of $\G$ are the orbits of $\G$ on $\QQ\cup\{\infty\}$.
Each cusp is represented by a rational number $\frac{a}{b}$, where 
$\gcd(a,b)=1$ and $b|N$. 
Moreover, $\frac{a}{b}$ and $\frac{a'}{b'}$ define the same cusp of $\G_0(N)$ if and only if $b=b'$ and $a\equiv a' \xmod \gcd(b,\frac{N}{b})$
(see e.g.\ \cite[Prop.~2.2.3]{MR1628193}).
The action of $W_N$ fuses the cusps represented by $\frac{a}{b}$ and $\frac{-ab}{N}$. 
The {\em infinite cusp} is represented by $\infty$ and by $\frac{1}{N}$.

Given $p$ prime and $b$ coprime to $p$ choose $b'$ such that $bb'\equiv 1\xmod p$, and set 
\begin{gather}\label{eqn:Xpb}
\Xp^b=V_p^{-1}T^bW_1T^{b'}V_p.
\end{gather}
Then $\G_0(p^2)\Xp^b$ is independent of the choice of $b$ and $b'$ mod $p$, and we have the following.
\begin{lem}\label{lem:Xp}
For $p$ prime and $b$ coprime to $p$, the infinite cusp of $\G_0(p^2)$ is mapped by $\Xp^b$ to $\frac bp$, 
and the $\G_0(p^2)\Xp^{b}$ for $0<b<p$ are the non-trivial right cosets of $\G_0(p^2)$ in $\G_0(p)$.
\end{lem}

Write $\HH$ for the complex upper half-plane.
For $\G=\G_0(N)$ or $\G=\G_0(N)^+$ the cusps of $\G$ naturally compactify the Riemann surface structure on the orbit space $\G\backslash\HH$
(see e.g.\ \cite[\S~1.3]{MR1291394}). 
The {\em genus} of $\G$ is the genus of this surface.
Using e.g.\ \cite[Prop. 1.40, 1.43]{MR1291394} it can be shown that
\begin{gather}\label{eqn:genusG0p}
	\genus(\G_0(p^v))
	=
	\begin{cases}
	0&\text{ if $v\leq 2$ and $p\in\{2,3\}$,}
	\\
	\frac{p-a}{12}&\text{ if $v=1$ and $p\equiv a \xmod 12$ for $a \in\{-1,5,7,13\}$,}
	\\
	\frac{p-a}{12}
	+
	\frac{(p-5)(p-1)}{12}
	&\text{ if $v=2$ and $p\equiv a \xmod 12$ for $a \in\{-1,5,7,13\}$.}
	\end{cases}
\end{gather}

In this work, a {\em modular function} is a holomorphic function on 
$\HH$ that is $\G_0(M)$-invariant for some $M>0$.
Given a modular function $f$, let $c_n(f)$ denote the coefficient of $q^n$ in the $q$-expansion of $f$, so that $f(\tau) = \sum_n c_n(f) q^n$.  
Here $q=\ex(\tau)$, for $\tau\in \HH$, where $\ex(x):=e^{2\pi i x}$.
The operator $U_N$ is defined by
\begin{equation}\label{eqn:Udef}
f|U_N = \frac{1}{N} \sum_{b=0}^{N-1} f|V_N^{-1}T^b,
\end{equation}
and satisfies 
$c_n(f|U_N) = c_{nN}(f)$.

The {\em Dedekind eta function} is $\eta(\tau):=q^{\frac{1}{24}}\prod_{n>0}(1-q^n)$.
Fix an integer $N>1$ and let $d=d_N$ be the least positive integer such that both 
$d(N-1)\equiv 0 \xmod{24}$ and $N^{d}$ is a square.  Then 
\begin{equation}\label{eqn:tN}
\jj_N(\tau) := \frac{\eta(\tau)^{d}}{\eta(N\tau)^{d}}
\end{equation}
is a modular function for $\G_0(N)$ (see e.g.\ \cite[Thm. 1.64]{MR2020489}),  
and the $q$-expansion of $\jj_N$ has the form
\begin{gather}\label{eqn:tNexp}
\jj_N(\tau) = q^{-n}\left(1-dq+\frac{d(d-3)}{2}q^2 + dq^N + \mathcal{O}(q^3)\right),
\end{gather}
where $d=d_N$ and $n=n_N:=\frac{d(N-1)}{24}$.

According to (\ref{eqn:tNexp}) the function $\jj_N$ has a simple pole at 
the infinite cusp of $\G_0(N)$ just for the $N$ such that $N-1$ divides $24$, namely,
\begin{gather}\label{eqn:nN=1}
N\in \{2,3,4,5,7,9,13,25\}.
\end{gather}
For the behavior of $t_N$ near non-infinite cusps we have the following.
\begin{lem}\label{lem:bdd}
For $N>1$, the function $\jj_N$ is bounded away from the infinite cusp of $\G_0(N)$ if and only if $N$ is prime, or the square of a prime.
\end{lem}
\begin{proof}
Let $b$ be a divisor of $N$ and let $a$ be an integer coprime to $b$. 
Using \cite[Thm.~1.65]{MR2020489} we find that the $q$-expansion of $t_N$ at the cusp represented by $\frac{a}{b}$ 
vanishes to order
\begin{gather}\label{eqn:tNord}
	\frac{d_N}{24\gcd(b,\frac{N}{b})}
	\left(\frac{N}{b}-b\right).
\end{gather}
In particular, $t_N$ is bounded near the cusp represented by $\frac{a}{b}$ if and only if $b\leq \frac{N}{b}$. 
The result now follows from the observation that $b^2\leq {N}$ for every proper divisor $b$ of $N$, if and only if $N$ is prime, or the square of a prime.
\end{proof}

From (\ref{eqn:tNexp}) and Lemma \ref{lem:bdd} we see that $n_N=1$, and $t_N$ is a Hauptmodul for $\G_0(N)$, and in particular $\G_0(N)$ has genus zero, when $N-1$ divides $24$. 
Using (\ref{eqn:genusG0p}) we obtain the following.
\begin{lem}\label{lem:G0pp2genuszero}
For $p$ prime, if $N=p$ or $N=p^2$ then $\G_0(N)$ has genus zero if and only if $N-1$ divides $24$. 
\end{lem}

The normalised Hauptmodul for $\G_0(1)$ is 
\begin{gather}\label{eqn:J1}
J_1(\tau)=j(\tau) - 744= q^{-1}+196884q+\mathcal{O}(q^2),
\end{gather}
while for $N$ such that $N-1$ divides $24$ the normalised Hauptmodul for $\G_0(N)$ is 
\begin{equation}\label{eqn:JNdef}
J_N = \jj_N - c_0(\jj_N)=\jj_N+d_N, 
\end{equation}
where $d_N = \frac{24}{N-1}$.
From (\ref{eqn:tNexp}) and (\ref{eqn:J1}-\ref{eqn:JNdef}) we infer that 
\begin{gather}\label{eqn:c1JN}
	c_1(J_1-J_N) = 
	\begin{cases}
		196884 - \frac{d(d-1)}{2} & \text{ for $N=2$,}\\
		196884 - \frac{d(d-3)}{2}  & \text{ for $N=3,4,5,7,9,13,25$,}
	\end{cases}
\end{gather}
where $d=d_N$. 
For future reference we record the $d_N$ for $N>1$ such that $n_N=1$, and the prime decompositions of the corresponding $c_1(J_1-J_N)$, in Table \ref{tab:c1J1JN}.

\begin{table}[h!] 
\caption{ 
The 
$d_N$ 
and 
$c_1(J_1-J_N)$,
for $N>1$ such that $n_N=1$.\label{tab:c1J1JN}}
\vspace{-6pt}
\begin{tabular}{c|ccccc}
\toprule
$N$ & $2$ & $3$ & $5$ & $7$ & $13$  \\ 
\midrule
$d_N$ & $24$ & $12$ & $6$ & $4$ & $2$  \\
$c_1(J_1-J_N)$ & $2^{16}\cdot 3$ & $2\cdot 3^9\cdot 5$ & $3^2\cdot 5^5\cdot 7$ & $2\cdot 7^4\cdot 41$ & $5\cdot 13^2\cdot 233$ \\ 
\bottomrule
\toprule
$N$ &  $2^2$ & $3^2$ & $5^2$ \\ 
\midrule
$d_N$ & $8$ & $3$ & $1$ \\
$c_1(J_1-J_N)$ & $2^{8}\cdot 769$ & $2^2\cdot 3^3\cdot 1823$ & $5\cdot 13^2\cdot 233$ \\ 
\bottomrule
\end{tabular}
\end{table}

From (\ref{eqn:WN}) 
and (\ref{eqn:tNexp})
and 
the formula
$\eta(W_1\tau)=\sqrt{-i\tau}\eta(\tau)$ 
we obtain that
\begin{equation}\label{eqn:sNdef}
s_N(\tau):=(\jj_N | W_N)(\tau) 
= N^\frac{d}{2} \jj_N(\tau)^{-1}
 = N^{\frac{d}{2}}q^n + \mathcal{O}(q^{n+1}),
\end{equation}
where $d=d_N$ and $n=n_N$. 
Thus $\jj_{N+}:=\jj_N + s_N$ is a modular function for $\G_0(N)^+$ with a pole of order $n_N$ 
at the infinite cusp. 
In particular, 
\begin{equation}\label{eqn:JN+def}
J_{N+} = \jj_N - c_0(\jj_N) + N^\frac{d}{2}\jj_N^{-1}
	= J_N + N^\frac{d}{2}\jj_N^{-1}
\end{equation}
is the normalised Hauptmodul for $\G_0(N)^+$, where $d=d_N$
(cf.\ (\ref{eqn:JNdef})).

\subsection{Level lowering}\label{sec:pre-inv}
We define and apply some level-lowering operators  
in this section.
Our first result demonstrates the level-lowering nature of $U_N$ (\ref{eqn:Udef}).
\begin{lem}\label{lem:levlowUp}
For $N$ a positive integer, 
if $f$ is modular for $\G_0(N^2)$ then $f|U_N$ is modular for $\G_0(N)$.
\end{lem}
\begin{proof}
Observe first that if $A=\left(\begin{smallmatrix}a&b\\ * &* \end{smallmatrix}\right)\in\G_0(N)$ then 
$T^{-ab}A\in\G_0(N,N):= V_N\G_0(N^2)V_N^{-1}$. 
Thus, since $f|U_N$ is $T$-invariant, it suffices to show that $f|U_N$ is $\G_0(N,N)$-invariant.
Next note that for $A=\left(\begin{smallmatrix}*&*\\ * &d \end{smallmatrix}\right)\in\G_0(N,N)$ and arbitrary $b$ we have 
$B_b:=V_N^{-1}T^bAT^{-bd^2}V_N \in \G_0(N^2)$.
Using this we compute
\begin{gather}\label{eqn:levlowUp}
	Nf|U_NA = 
	\sum_{b\xmod N} f|V_N^{-1}T^bA
	= \sum_{b\xmod N} f|B_bV_N^{-1}T^{bd^2}
	=\sum_{b\xmod N} f|V_N^{-1}T^b,
\end{gather}
where the last equality in (\ref{eqn:levlowUp}) holds because 
$B_b\in\G_0(N^2)$ and 
$d\equiv 1 \xmod N$.
We conclude from (\ref{eqn:Udef}) that $f|U_NA=f|U_N$.
\end{proof}

\begin{rmk}
A similar computation to (\ref{eqn:levlowUp}) shows that if $f\in \G_0(MN^2)$ then $f|U_N\in\G_0(MN)$, for arbitrary positive integers $M$ and $N$.
\end{rmk}

Recall $X_p^b$ from (\ref{eqn:Xpb}).
For $p$ prime and $f$ a modular function for $\G_0(p^2)$ we now define 
\begin{gather}\label{eqn:Xp}
f|\Xp:=f + \sum_{0<b<p}f|\Xp^b.
\end{gather} 
Then Lemma \ref{lem:Xp} shows that 
$f|\Xp$ is well-defined, and also proves the following. 
\begin{lem}\label{lem:levlowXp}
For $p$ prime, if $f$ is modular for $\G_0(p^2)$ then $f|\Xp$ is modular for $\G_0(p)$.
\end{lem}

Next we use $U_p$ and $W_p$ together, to lower levels from $\G_0(p)$ to $\G_0(1)$. 
\begin{lem}\label{lem:levlowUpWp}
If $f$ is 
modular 
for $\G_0(p)$ then $pf|U_p+f|W_p$ is 
modular 
for $\G_0(1)$.
\end{lem}
\begin{proof}
Set $h:=pf|U_p+f|W_p$. 
Then $h$ is $T$-invariant because $W_p$ normalises $\G_0(p)$.
To show that $h$ is $W_1$-invariant
we write $h=h_1+h_2$, where 
$h_1 = \sum_{0<b<p} f|V_p^{-1}T^b$
and
$h_2 = f|V_p^{-1} + f|W_p$. 

Now $h_2$ is $W_1$-invariant because $W_p=V_p^{-1}W_1$ (cf.\ (\ref{eqn:WN})). 
For $h_1$ we have
\begin{gather}\label{eqn:levlowUpWp}
	h_1|W_1 = \sum_{0<b<p} f|V_p^{-1}T^bW_1 = \sum_{0<b<p}f|\Xp^bV_p^{-1}T^{-b'}=\sum_{0<b<p}f|V_p^{-1}T^{-b'}
\end{gather}
(see (\ref{eqn:Xpb})).
We conclude that $h_1|W_1=h_1$ since the assignment $b\mapsto -b'$ defines a bijection on the non-zero elements of $\ZZ\xmod p$.
\end{proof}

\begin{rmk}
A similar computation to (\ref{eqn:levlowUpWp}) shows that if $f\in \G_0(Mp)$ for $M$ a positive integer coprime to $p$, and if $W_p$ is taken to be a suitably chosen Atkin--Lehner involution for $\G_0(Mp)$, then $pf|U_p+f|W_p$ belongs to $\G_0(M)$.
\end{rmk}

The next result is essentially \cite[Lem.~2.1]{MR3435726}. It can also be found in \cite[p.~318]{MR554399}
\begin{lem}\label{lem:levlowJp+}
For $p$ prime such that $\G_0(p)^+$ has genus zero,  
$J_1-J_{p+} = p J_{p+}|U_p$. 
\end{lem}
\begin{proof}
Lemma \ref{lem:levlowUpWp} tells us that $pJ_{p+}|U_p+J_{p+}|W_p$ is $\G_0(1)$-invariant, and therefore a polynomial in $J_1$. 
Since 
$J_{p+}|U_p=\mathcal{O}(q)$ and $J_{p+}|W_p=J_{p+}=q^{-1}+\mathcal{O}(q)$, 
we have $pJ_{p+}|U_p+J_{p+}=J_1$.
\end{proof}

Lemma \ref{lem:levlowJp+} has the following consequence.
\begin{lem}\label{lem:vJ1vJpp}
For $k\geq 0$, and $p$ prime such that $\G_0(p)^+$ has genus zero, 
$\val_p(J_1|U_p^k)=\val_p(J_{p+}|U_p^k)$.
\end{lem}
\begin{proof}
By Lemma \ref{lem:levlowJp+} we have $J_1|U_p^k =J_{p+}|U_p^k + pJ_{p+}|U_p^{k+1}$. The claim follows because $\val_p(a+b)=\val_p(a)$ when $\val_p(a)<\val_p(b)$, and 
$\val_p(f)<\val_p(pf|U_p)$ 
for any modular function $f$.
\end{proof}

The following two lemmas will be used to constrain $J_{p^2}$.

\begin{lem}\label{lem:ptp2kUpconst}
For $p$ prime such that $\G_0(p)$ has genus zero, if $m$ is such that $0<m{n_{p^2}}<{p}$, then $pt_{p^2}^m|U_p$ is constant.
\end{lem}
\begin{proof}
Set $f:=pt_{p^2}^m|U_p$.
Then $f$ is modular for $\G_0(p)$ according to Lemma \ref{lem:levlowUp}, and bounded near the infinite cusp according to the condition on $m$ (see (\ref{eqn:tNexp})).
By Lemma \ref{lem:G0pp2genuszero} we have that $s_p$ (see (\ref{eqn:sNdef})) is a Hauptmodul for $\G_0(p)$. 
Since $s_p$ is bounded near the infinite cusp we have
$f=P(s_p)$ for some polynomial $P(x)$.
Consider $f|W_p$. 
On the one hand $f|W_p = P(s_p)|W_p = P(t_p)$, so since $n_p=1$, the degree of $P(x)$ is the order of the pole in the $q$-expansion of $f|W_p$. 
On the other hand
\begin{gather}\label{eqn:ptp2kUpconst}
	f|W_p
	= t_{p^2}^m|V_p^{-1}W_p + \sum_{0<b<p} t_{p^2}^m|V_p^{-1}T^bW_p
	= s_{p^2}^m + \sum_{0<b<p}  t_{p^2}^m|\Xp^b V_{p}^{-1}T^{-b'}V_p
\end{gather}
(see (\ref{eqn:Xpb}) and (\ref{eqn:Udef}), and cf.\ (\ref{eqn:levlowUpWp})). 
Now $t_{p^2}^m|\Xp^b$, for $0<b<p$, is the expansion of $t_{p^2}^m$ at a non-infinite cusp of $\G_0(p^2)$ according to Lemma \ref{lem:Xp},
so $t_{p^2}^m|\Xp^b$ is bounded near the infinite cusp according to Lemma \ref{lem:bdd}.
Since 
$q|V_{p}^{-1}T^{-b'}V_p$ is simply 
$q$ times a $p$-th root of unity, 
no summand 
in (\ref{eqn:ptp2kUpconst}) 
has a pole in its $q$-expansion.
Thus $P(x)$ has degree $0$, and $f$ is constant. 
\end{proof}

\begin{lem}\label{lem:cnJp2}
For $p=2$ and $p=3$ we have
$c_n(J_{p^2})=0$ unless 
$n\equiv -1\xmod p$.
\end{lem}
\begin{proof}
To ease notation set $\cns n = c_n(J_{p^2})$.
Then $\cns n = c_n(t_{p^2})$ for $n>0$ by (\ref{eqn:JNdef}).
Since $n_{p^2}=1$ we may apply Lemma \ref{lem:ptp2kUpconst} with $0<m<p$.
Taking $m=1$ we obtain $\cns n =0$ for $n\equiv  0\xmod p$. 
This proves the claim for $p=2$.
For $p=3$ we need to additionally show that $\cns n=0$ for $n>0$ such that $n\equiv 1\xmod 3$. 
Since $d_9=3$ (see Table \ref{tab:c1J1JN}) we have 
$J_9^2 = t_9^2 + 6t_9  + 9$.
Applying Lemma \ref{lem:ptp2kUpconst} with 
$m=1$ and $m=2$ we thus obtain $c_{3n}(J_9^2) =0$ for $n>0$. 
It follows that
\begin{gather}\label{eqn:c3nJ92}
	2\cns{3n+1} =-\sum_{k=0}^{n-1}\left(  \cns{3k+1}\cns{3(n-k-1)+2} + \cns{3k+2}\cns{3(n-k-1)+1} \right)
\end{gather}
for $n>0$.
We have $\cns{1} =0$ by (\ref{eqn:tNexp}). The vanishing of $\cns{n}$ for $n>0$ such that $n\equiv 1\xmod 3$ now follows by induction on $n$ using (\ref{eqn:c3nJ92}).
\end{proof}

The remaining results in this section are used in the computation of 
$\val_p(J_1-J_p)$ (see Proposition \ref{pro:J1Jp}) and 
$\val_p(J_1-J_{p^2})$ (see Proposition \ref{pro:J1Jp2}).
We note that one may simply use the first few coefficients of a modular function to determine an upper bound on its $p$-divisibility. 
As mentioned in \S~\ref{sec:int}, we prefer to take an approach which highlights the underlying structure of these objects, with an eye toward a more general theory.

\begin{pro}\label{pro:valpJ1Jp}
For $p$ prime such that $\G_0(p)$ has genus zero,
$\val_p(J_1-J_p)=\val_p(c_1(J_1-J_p))$.
\end{pro}
\begin{proof}
We have $\val_p(J_1-J_p)\leq \val_p(c_1(J_1-J_p))$ by definition. 
For the reverse inequality, 
note first that $J_1-J_p$ is bounded near the infinite cusp of $\G_0(p)$, so, as in the proof of Lemma \ref{lem:ptp2kUpconst},
there is a polynomial $P(x)$, 
with vanishing constant term, such that $J_1-J_p = P(s_p)$.
Writing
\begin{gather}
(J_1-J_p)|W_p = P(t_p)=q^{-p}-d_p+\mathcal{O}(q)
\end{gather} 
we obtain that $P(x)+d_p$ is the $p$-th {\em Faber polynomial} for $t_p$ (see e.g.\ \cite[\S~3]{MR2904095}), and conclude
that $P(x)$ has integer coefficients. 

Now write $P(x)=a_1x+P_*(x)x^2$, where $P_*(x)\in\ZZ[x]$.
Then $\val_p(J_1-J_p)=\val_p(P(s_p))$ is bounded below by the minimum of $\val_p(a_1s_p)$ and $\val_p(P_*(s_p)s_p^2)$.
Applying (\ref{eqn:sNdef}) we obtain 
$\val_p(c_1(J_1-J_p)) = \val_p(c_1(a_1s_p))=\val_p(a_1s_p)$ 
and also 
$\val_p(P_*(s_p)s_p^2)\geq \val_p(s_p^2) = d_p$.
Thus 
$\val_p(J_1-J_p)$ is bounded below by the minimum of $\val_p(c_1(J_1-J_p))$ and $d_p$.
We check that $d_p\geq \val_p(c_1(J_1-J_p))$
using Table \ref{tab:c1J1JN}.
\end{proof}

\begin{pro}\label{pro:valpJpJp2}
For $p$ prime such that $\G_0(p^2)$ has genus zero,
$\val_p(J_p-J_{p^2}) = d_{p^2}$.
\end{pro}
\begin{proof}
According to Lemma \ref{lem:G0pp2genuszero} we have $p\in \{2,3,5\}$.
For $p=5$ we have $d_{p^2}=1$, and the claim follows from (\ref{eqn:tN}) and (\ref{eqn:JNdef}), and the fact that $(1-x)^5 \equiv 1-x^5\xmod 5$.
For $p\in \{2,3\}$ we consider
$h=  t_{p^2}|\Xp$ (see (\ref{eqn:Xp})),
which is 
modular for $\G_0(p)$ by Lemma \ref{lem:levlowXp}.
By Lemma \ref{lem:Xp} we have that 
$h$ is $q^{-1}+\mathcal{O}(1)$, and a Hauptmodul for $\G_0(p)$, so $h=J_p+c_0(h)$.
Next observe that 
\begin{gather}\label{eqn:XpbTfrac}
	\Xp^b = T^{\frac bp}W_{p^2}T^{\frac {b'}p}
\end{gather}
(see (\ref{eqn:Xpb})),
where 
$T^x := \left(\begin{smallmatrix}1&x\\0&1\end{smallmatrix}\right)$.
For the rest of this proof set $d=d_{p^2}$. By (\ref{eqn:JNdef}) 
and Lemma \ref{lem:cnJp2} we have 
\begin{gather}
t_{p^2}|T^{\frac bp} = \ex(- \tfrac bp)t_{p^2}+(\ex(-\tfrac{b}{p}) - 1)d.
\end{gather}
Also, 
$b^2\equiv 1\xmod p$ 
for $p\in \{2,3\}$, so we may assume that $b'=b$ in (\ref{eqn:XpbTfrac}). 
Thus we have
\begin{gather}\label{eqn:tp2Xpb}
t_{p^2}|\Xp^b
= \ex(- \tfrac bp )p^{d}t_{p^2}^{-1}|T^{\frac b p}+(\ex(- \tfrac bp )-1)d.
\end{gather}

In case $p=2$ we deduce from (\ref{eqn:tp2Xpb}) that
\begin{gather}
h
= 
t_{4} + t_{4}|X_2^1
= J_{4} - 3d+ 2^{d}\sum_{n>0} (-1)^{n+1}c_n(t_4^{-1})q^n.
\end{gather}
Using $h=J_2+c_0(h)$ we get
$J_2 - J_4 = 2^d\sum_{n>0}(-1)^{n+1}c_n(t_4^{-1})q^n$. In particular, the claim holds for $p=2$. 

In case $p=3$ we deduce from (\ref{eqn:tp2Xpb}) that
\begin{gather}
h
= 
t_{9} + t_{9}|X_3^1 + t_{9}|X_3^2
= J_{9} - 4d  + 
3^d\sum_{n>0}2\Re(\ex(\tfrac{n-1}3))c_n(t_{9}^{-1})q^n.
\end{gather}
Using $h=J_3+c_0(h)$ we now get
 $J_3 - J_9 = 3^d\sum_{n>0}2\Re(\ex(\tfrac{n-1}3))c_n(t_{9}^{-1})q^n$, 
 and the claim holds for $p=3$. 
\end{proof}

%--------------------------------------------------------------------------------------------------------------------------------%
\section{Application of Deligne's Theorem}\label{sec:Deligne}
%--------------------------------------------------------------------------------------------------------------------------------%

For convenience, in this section we redefine $\gt{S}_p\subset\FF_{p^2}$ (cf.\ (\ref{eqn:gtSp1gtSp2})) to be the set of supersingular $J_1$-values in characteristic $p$.
I.e., $\gt{S}_p$ denotes the values that arise as $\mt_{1}(E)= j(E)-744$ for $E$ a supersingular elliptic curve in characteristic $p$. 
Suppose that $\Gamma_0(p)^+$ has genus zero. Then by 
\cite{MR417184} we have that $\gt{S}_p\subset\FF_p$, so we may for each $\a \in \gt{S}_p$ choose an integer $\hat\a$ such that $\hat\a\equiv \a\xmod p$. 
The following is obtained using a theorem of Deligne, as described by Dwork \cite{MR0294346} and Koike \cite{MR0437461}.

\begin{pro}\label{pro:T1Up}
For $p$ prime such that $\Gamma_0(p)^+$ has genus zero there exist integers $\hat\a\equiv\a\xmod p$, for each $\a\in \gt{S}_p$, such that
\begin{gather}\label{eqn:T1Up}
p\mt_{1}|U_p = -\sum_{\a\in\gt{S}_p}\sum_{n\geq 1}A_{n}(\hat\a)(\mt_{1}-\hat\a)^{-n},
\end{gather}
where the $A_{n}(\hat\a)$ are integers such that 
\begin{gather}\label{eqn:T1Up-nup}
	\val_p(A_{n}(\hat\a))\geq 
	\begin{cases}
	\frac{3np+1}{p+1}&\text{ if $\a\equiv -744\xmod p$,}\\
	\frac{2np+1}{p+1}&\text{ if $\a\equiv 984\xmod p$,}\\
	\frac{np+1}{p+1}&\text{ else.}
	\end{cases}
\end{gather}
Moreover, if $p>3$ then the bounds in (\ref{eqn:T1Up-nup}) are sharp for $n=1$.
\end{pro}
\begin{proof}
Let $\Phi_p$ denote the $p$-th Faber polynomial for $\mt_1$, being the unique polynomial such that $\Phi_p(J_1)=q^{-p}+\mathcal{O}(q)$ (see e.g. \cite[\S~3]{MR2904095}),
and observe that $\Phi_p(\mt_1) = j_p$, where $j_p$ is as in \cite{MR2048229}.
Then from the results of \cite[\S~7]{MR0294346} (see also \cite[(36-37)]{MR0437461}), and \cite[Theorem 1.1]{MR2048229}, we have
\begin{gather}\label{eqn:Deligne-T1Vp}
\mt_1|V_p =\Phi_p(\mt_1) +
\sum_{\a\in\gt{S}_p}\sum_{n\geq 1}A_{n}(\hat\a)(\mt_{1}-\hat\a)^{-n},
\end{gather}
with 
$A_{n}(\hat\a)$ as in 
(\ref{eqn:T1Up-nup}).
(Note that stronger bounds are available for $p\in\{2,3\}$. See \cite[p.~89]{MR0294346} and \cite[(38)]{MR0437461}.
Also, the $A_{n}(\hat\a)$ are denoted $A_n^{(i)}$ in \cite[\S~7]{MR0294346}.)
Applying Lemma \ref{lem:levlowUpWp} to $\mt_1$ we obtain that $J_1|W_p + pJ_1|U_p = J_1|V_p + pJ_1|U_p$ is $\G_0(1)$ invariant, and therefore a polynomial in $J_1$. 
This function is also $q^{-p}+\mathcal{O}(q)$, so 
$J_1|V_p + pJ_1|U_p = \Phi_p(J_1)$. This proves (\ref{eqn:T1Up}-\ref{eqn:T1Up-nup}).

For the sharpness statement, Theorem 8.2 of \cite{MR0294346}
tells us that $\val_p(A_1(\hat\a))=1$ when $p>3$ and $\a\not\equiv -744,984\xmod p$.
As remarked at the beginning of \cite[\S~7.e]{MR0294346}, the counterparts to Theorem 8.2 for $\a\equiv -744,984\xmod p$ 
may be obtained by noting that $J_1+744$ behaves like $(\lambda-\lambda_0)^3$ near $J_1 = -744$, and $J_1 - 984$ behaves like $(\lambda-\lambda_0)^2$ near $J_1 = 984$, where
$\lambda$ denotes the modulus defined by the Legendre normal form 
(see \cite[\S~4]{MR0294346}).
The result is that 
$\val_p(A_1(\hat\a))=3$ when $\a\equiv  -744\xmod p$,
and 
$\val_p(A_1(\hat\a))=2$ when $\a\equiv  984\xmod p$.
This completes the proof.
\end{proof}

We list the sets $\gt{S}_p$ for $p$ such that $\Gamma_0(p)^+$ has genus zero explicitly in Table \ref{tab:ssm=1}.  From the equivalent result for $j(E)$ it follows that $-744=0-744$ belongs to $\gt{S}_p$ if and only if $p\equiv 2 \xmod 3$, and $984=1728-744$ belongs to $\gt{S}_p$ if and only if $p\equiv 3\xmod 4$. In Table \ref{tab:ssm=1} we present representatives $\hat\a$ for the values $\a\in \gt{S}_p$ in the range $0\leq \hat\a<p$, where the first and second columns indicate $\hat\a\equiv -744 \xmod p$ and $\hat\a\equiv 984 \xmod p$, respectively.

\begin{table}[h!] 
\caption{Representatives of supersingular $J_1$-values, for $p$ such that $\mathfrak{S}_p\subset\FF_p$.\label{tab:ssm=1}}
\vspace{-6pt}
\begin{tabular}{c||c|c|c}
\toprule
$p$ & -744 & 984 & other s.s.~$J_1$-values \\ 
\midrule
2 & 0 & 0 & - \\ 
3 & 0 & 0 & - \\ 
5 & 1 & - & - \\ 
7 & - & 4 & - \\ 
11 & 4 & 5 & - \\ 
13 & - & - & 2 \\ 
17 & 4 & - & 12 \\ 
19 & - & 15 & 4 \\ 
23 & 15 & 18 & 11 \\ 
29 & 10 & - & 6, 12 \\ 
31 & - & 23 & 2, 4 \\ 
41 & 35 & - & 22, 26, 38 \\ 
47 & 8 & 44 & 5, 17, 18 \\ 
59 & 23 & 40 & 11, 12, 38, 51 \\ 
71 & 37 & 61 & 6, 7, 14, 32, 54 \\ 
\bottomrule
\end{tabular}
\end{table}

Fix a prime $p$ such that $\Gamma_0(p)^+$ has genus zero, let $\alpha_1, \ldots, \alpha_\ell$, 
for $\ell=\ell_p\geq 0$, denote the supersingular $J_1$-values not in $\{-744,984\}$, and let $\beta_0$ and $\beta_1$ denote the supersingular $J_1$-values $-744$ and $984$, respectively, noting that any of these may or may not exist for a given prime $p$.  
From the bounds on $\val_p(A_{n}(\hat\a))$ in Proposition \ref{pro:T1Up} it follows that 
\begin{equation}\label{eqn:J_1Upmp}
pJ_1 | U_p \equiv \sum_{i=1}^\ell \frac{A_{1}(\hat\a_i)}{J_1-\hat\a_i} \xmod{p^2},
\end{equation}
with $\val_p(A_1(\hat\a_i))=1$ for each $i$, and furthermore that 
\begin{equation}\label{eqn:J_1Upmp^2}
pJ_1 | U_p \equiv \frac{A_{1}(\hat\beta_1)}{J_1-\hat\beta_1}  + \sum_{i=1}^\ell \left( \frac{A_{1}(\hat\a_i)}{J_1-\hat\a_i} + \frac{A_{2}(\hat\a_i)}{(J_1-\hat\a_i)^2}  \right) \xmod{p^3},
\end{equation}
where the first summand on the right-hand side of (\ref{eqn:J_1Upmp^2}) is understood to be $0$ in case $\beta_1$ does not exist (cf.\ Table \ref{tab:ssm=1}), and the summation over $i$ is understood to be zero in case $\ell=\ell_p=0$.

From Table \ref{tab:ssm=1} we see that $\beta_1$ exists, and $\ell_p=0$, for $p=11$.
From this together with (\ref{eqn:T1Up-nup}) and \eqref{eqn:J_1Upmp^2} we
directly infer the following.
\begin{pro}\label{pro:p11}
For $p=11$, we have
$\val_p(pJ_1 | U_p) =2$.
\end{pro}

Next we use \eqref{eqn:J_1Upmp^2} to determine $\val_p(pJ_1| U_p)$ for $p\geq 13$.
\begin{pro}\label{pro:geq13}
For $p\geq 13$ prime such that $\Gamma_0(p)^+$ has genus zero, 
$\val_p(pJ_1 | U_p) =1$.
\end{pro}
\begin{proof}
We see from Table \ref{tab:ssm=1} that 
$\ell=\ell_p>0$.  
Working in $\ZZ[[q]]$, write
\begin{gather}
\frac{1}{J_1-\hat\a_i} = \frac{J_1^{-1}}{1-\hat\a_i J_1^{-1}} = J_1^{-1}\sum_{n\geq 0}(\hat\a_iJ_1^{-1})^n.
\end{gather}
Since $J_1^{-n} = q^n + \mathcal{O}(q^{n+1})$, we have from \eqref{eqn:J_1Upmp} that
\begin{gather}
\begin{split}
pJ_1 | U_p &\equiv \sum_{n\geq 1} \sum_{i=1}^\ell A_1(\hat\a_i) \hat\a_i^{n-1} J_1^{-n} \xmod{p^2} \\
&\equiv  \left[ \sum_{i=1}^\ell A_1(\hat\a_i) \hat\a_i^0 \right](q +\cdots) + \left[ \sum_{i=1}^\ell A_1(\hat\a_i)  \hat\a_i^1 \right](q^2 +\cdots) + \cdots \xmod{p^2}.
\end{split}
\end{gather}
If it were the case that $J_1 | U_p \equiv 0 \xmod{p}$, it would follow that
\begin{equation}\label{eq3}
\sum_{i=1}^\ell A_1(\hat\a_i) \equiv \sum_{i=1}^\ell A_1(\hat\a_i)  \hat\a_i \equiv \sum_{i=1}^\ell A_1(\hat\a_i)  \hat\a_i^2 \equiv \cdots \equiv 0 \xmod{p^2}.
\end{equation}
However, 
\begin{gather}
M = \left[\begin{array}{cccc}\hat\a_1^0 & \hat\a_1^1 & \cdots & \hat\a_1^{\ell-1} \\
\vdots & \vdots & \vdots & \vdots \\ 
\hat\a_1^0 & \hat\a_\ell^1 & \cdots & \hat\a_\ell^{\ell-1} \end{array}\right]
\end{gather}
is a Vandermonde matrix and $\hat\a_1, \ldots, \hat\a_\ell$ are distinct, thus $M$ is invertible.  
If \eqref{eq3} held then 
$(A_1(\hat\a_1), \cdots, A_1(\hat\a_\ell))^T$ 
would belong to the kernel of $M^T$ modulo $p^2$,
which would contradict the fact that $\val_p(A_1(\hat\a_i))=1$ for each $i$.  
Thus $\val_p(pJ_1 | U_p) =1$, as desired. 
\end{proof}

%--------------------------------------------------------------------------------------------------------------------------------%
\section{Proofs of Theorems \ref{thm:main} and \ref{thm:main-3d}}\label{sec:proof}
%--------------------------------------------------------------------------------------------------------------------------------%

We begin this section with propositions that determine $\val_p(J_1 - J_{p})$ and $\val_p(J_1 - J_{p+})$ in the case that $\Gamma_0(p)$ has genus zero, and determine
$\val_p(J_1 - J_{p^2})$ in the case that $\Gamma_0(p^2)$ has genus zero.

\begin{pro}\label{pro:J1Jp}
Let $p$ prime such that $\G_0(p)$ has genus zero.  
Then 
(\ref{eqn:wppgenus0}) 
and 
(\ref{eqn:wpgenus0}) 
both hold.
\end{pro}
\begin{proof}
In light of Lemma \ref{lem:G0pp2genuszero}, the primes in play are those $p$ such that $p-1$ divides $12$.
Using Table \ref{tab:c1J1JN} we check that
$\val_p(c_1(J_1-J_p))$ agrees with the right-hand side of (\ref{eqn:wpgenus0}) for these primes.
Proposition \ref{pro:valpJ1Jp} then implies (\ref{eqn:wpgenus0}).
In particular, we have 
\begin{gather}\label{eqn:valpJ1JpvalpcJaJpfracd2}
\val_p(J_1-J_p) = \val_p(c_1(J_1-J_p))>\tfrac{12}{p-1}.
\end{gather}

Set $d=d_p$. Then for (\ref{eqn:wppgenus0}) we require to check that $\val_p(J_1-J_{p+})=\frac{d}{2}$.
Since $J_{p+}\equiv J_{p}\xmod p^{\frac{d}{2}}$ according to (\ref{eqn:JN+def}), it follows from (\ref{eqn:valpJ1JpvalpcJaJpfracd2}) that 
$\val_p(J_1-J_{p+})\geq \tfrac{d}{2}$.
For the reverse inequality we apply (\ref{eqn:JN+def}) again to obtain $c_1(J_1-J_{p+})=c_1(J_1-J_p)-p^{\frac d2}$, and deduce using (\ref{eqn:valpJ1JpvalpcJaJpfracd2}) (cf.\ Lemma \ref{lem:vJ1vJpp})
that $\val_p(c_1(J_1-J_{p+}))= \frac d2$. 
The claim follows because $\val_p(J_1-J_{p+})\leq \val_p(c_1(J_1-J_{p+}))$ by definition.
\end{proof}

\begin{pro} \label{pro:J1Jp2}
Let $p$ prime such that $\G_0(p^2)$ has genus zero.  
Then (\ref{eqn:wp2genus0}) holds.
\end{pro}
\begin{proof}
Lemma \ref{lem:G0pp2genuszero} tells us that $p\in\{2, 3, 5\}$, and in particular, $p^2-1$ divides $24$. 
Thus $d_{p^2}=\frac{24}{p^2-1}$, and we require to show that $\val_p(J_1-J_{p^2}) = d_{p^2}$ (cf.\ (\ref{eqn:tN})).
This follows from the properties of $\val_p$ (cf.\ Lemma \ref{lem:vJ1vJpp}), because 
Proposition \ref{pro:J1Jp} gives us $\val_p(J_1-J_p)>d_{p^2}$, 
and 
Proposition \ref{pro:valpJpJp2} tells us that $\val_p(J_p-J_{p^2})= d_{p^2}$.
\end{proof}

We are now ready to prove our main results.
\begin{proof}[Proof of Theorem \ref{thm:main}]
First suppose that $\Gamma_0(p)^+$ has positive genus. 
Then, by our conventions, $J_{p}$, $J_{p+}$ and $J_{p^2}$ are all zero, so each summand on the right-hand side of (\ref{eqn:main}) is $\val_p(J_1)=0$.
This proves Theorem \ref{thm:main} for $p$ such that $\Gamma_0(p)^+$ does not have genus zero.

Next consider the case that $\Gamma_0(p)^+$ has genus zero but $\Gamma_0(p)$ has positive genus.
I.e., $p$ divides $\#\MM$ (\ref{eqn:ord}), and either $p=11$ or $p>13$ (cf.\ Lemma \ref{lem:G0pp2genuszero}).
Then the right-hand side of (\ref{eqn:main}) reduces to $\val_p(J_1 - J_{p+})$, 
which is 
$\val_p(pJ_1|U_p)$ according to Lemmas \ref{lem:levlowJp+} and \ref{lem:vJ1vJpp}.
Proposition \ref{pro:p11} gives us $\val_p(pJ_1|U_p)=\val_p(\#\MM)$ for $p=11$, and 
Proposition \ref{pro:geq13} handles $p>13$.

It remains to consider the primes $p>3$ for which $\G_0(p)$ has genus zero, i.e., $p\in\{5,7,13\}$ (cf.\ Lemma \ref{lem:G0pp2genuszero}). 
For $p\in\{7,13\}$ the group $\Gamma_0(p^2)$ has positive genus (cf.\ Lemma \ref{lem:G0pp2genuszero}), and the right-hand side of (\ref{eqn:main}) reduces to 
$\val_p(J_1-J_{p+}) + \val_p(J_1 - J_p)$. The desired equalities then follow from Proposition \ref{pro:J1Jp}. 
Propositions 
\ref{pro:J1Jp}-\ref{pro:J1Jp2} handle $p=5$. 
This completes the proof.
\end{proof}

\begin{proof}[Proof of Theorem \ref{thm:main-3d}]
By Ogg's work \cite{MR417184}, the group $\Gamma_0(p)^+$ has positive genus exactly when $\gt{S}_p^2$ is nonempty. 
The case that $\Gamma_0(p)^+$ has genus zero while $\Gamma_0(p)$ has positive genus is exactly the case that 
$\#\gt{S}_p = \#\gt{S}_p^1>1$, and $\Gamma_0(p)$ has genus zero just when $\#\gt{S}_p = \#\gt{S}_p^1=1$.
Thus (\ref{eqn:main-3d}) correctly evaluates to $0$ 
for the primes $p$ such that $\Gamma_0(p)^+$ has positive genus.  
Moreover, from Table \ref{tab:ssm=1} and \cite[III.\ Thm.~10.1]{MR2514094} 
we have
\begin{gather}\label{eqn:moap}
\moa_p=
\begin{cases}
2 & \text{ if $p\in \{ 13, 17, 19, 23, 29, 31, 41, 47, 59, 71\}$,} \\
4 & \text{ if $p\in \{7,11\}$,} \\
6 & \text{ if $p=5$,}\\
12 & \text{ if $p=3$,}\\
24 & \text{ if $p=2$.}
\end{cases}
\end{gather}
Thus $\val_p(\#\MM)=\frac12 \moa_p$ for primes $p>3$ such that $\Gamma_0(p)^+$ has genus zero and $\Gamma_0(p)$ has positive genus, while $\val_p(\#\MM)=\frac32 \moa_p$ for primes $p>3$ such that $\Gamma_0(p)$ has genus zero, as desired.
\end{proof}

%--------------------------------------------------------------------------------------------------------------------------------%

%--------------------------------------------------------------------------------------------------------------------------------%

\end{document}